\RequirePackage{fix-cm}
\documentclass[10pt]{amsart} 
\usepackage{amsmath}
\usepackage{amssymb}
\usepackage{mathrsfs}
\usepackage[usenames,dvipsnames]{pstricks}
 \usepackage{epsfig}
 \usepackage{pst-grad} 
 \usepackage{pst-plot} 
 \usepackage[space]{grffile} 
 \usepackage{etoolbox} 
 \makeatletter 
 \patchcmd\Gread@eps{\@inputcheck#1 }{\@inputcheck"#1"\relax}{}{}
 \makeatother

\usepackage{graphicx}
\usepackage{dsfont}
\usepackage{url}
\usepackage{stackrel}
\usepackage{caption}

\newtheorem{theorem}{Theorem}[section]
\newtheorem{lemma}[theorem]{Lemma}
\newtheorem{proposition}[theorem]{Proposition}

\newtheorem{remark}[theorem]{Remark}

\newcommand{\vG}{G}
\newcommand{\ball}{\mathbb{D}}
\newcommand{\eC}{\mathscr{C}}
\newcommand{\R}{\mathbb{R}}

\newcommand{\Z}{\mathbb{Z}}
\newcommand{\Prob}[1]{\mathbb{P}\left(#1\right)}
\newcommand{\PC}[2]{\mathbb{P}\left(#1\left|#2\right)\right.}
\newcommand{\abs}[1]{\left|#1\right|}
\newcommand{\lL}{\left\{}
\newcommand{\rL}{\right\}}
\newcommand{\lP}{\left(}
\newcommand{\rP}{\right)}
\newcommand{\lC}{\left[}
\newcommand{\rC}{\right]}
\newcommand{\dist}[2]{\mbox{dist}\left(#1,#2\right)}
\newcommand{\integral}[4]{\int_{#1}^{#2} \int_{#3}^{#4} r_A^3 dr_A d\beta_A}

\begin{document}

\title{ Angle distribution of two random chords\\ in the disc: A sine law}

\author{Jesus Igor Heberto Barahona Torres}
\address{
Unidad Cuernavaca del Instituto de Matem\'aticas. Universidad Nacional Aut\'onoma de M\'exico.}
\email{igor@im.unam.mx}

\author{Paulo Cesar Manrique-Mir\'on}
\address{
C\'atedra CONACyT -- Unidad Cuernavaca del Instituto de Matem\'aticas. Universidad Nacional Aut\'onoma de M\'exico.}
\email{paulo.manrique@im.unam.mx}

\author{Erick Trevi\~{n}o-Aguilar}
\address{
Unidad Cuernavaca del Instituto de Matem\'aticas. Universidad Nacional Aut\'onoma de M\'exico.}
\email{erick.trevino@im.unam.mx }

\date{\today}

\maketitle
\markboth{Angle distribution of two random chords in the disc: A sine law}{Igor Barahona, Paulo Manrique, Erick Trevi\~no}

\begin{abstract}
Motivated by models in engineering and also biology we determine in closed form the probability density function of the angle shaped by two random chords in a fixed disc. Our main result focus on the event in which the intersection locates inside the fixed disc and establishes a sine law.
\end{abstract}

\section{Introduction}
\label{intro}
The beginning of Geometric Probability can be traced back at least to Buffon's needle problem posed by Georges--Louis Leclerc, Comte de Buffon \cite{1997introduction,Seneta2001,solomon1978geometric}. This problem asks for the probability that a needle of length $l$ dropped at random intersects a set of parallel lines distributed uniformly at a fixed distance $d$ from each other. The solution depends on $d$ and $l$ and in the  special case  $d>l$  the probability is given by  $\frac{2l}{\pi d}$. Although this was not the original motivation, the solution has been used since then as physical mechanism to approximate $\pi$'s value. 

Recent developments on computer science and wireless communications have opened a new niche for studies related to Geometric Probability with applications in several fields such as biology, engineering, and transportation; see e.g.,  \cite{8362874,7996566,Orozco}. In particular, the so--called random waypoint model commonly appears in  wireless communication networks studies; see e.g., \cite{Bettstetter2004}. The analysis of this random model usually has been done by simulations since closed formulas for probability densities are usually unavailable, and some papers investigate the performance of such numerical approach. Thus, models in which closed formulas are obtained for probability densities are interesting. We give a few examples where such formulas exist, before presenting the main contribution of the paper.\\

A core object analyzed in Geometric Probability is the random angle between different random objects as in Morton's generalization of Buffon's needle problem \cite{carleton2015microscale,cloherty2016sensory,Laurent-Gengoux2019,morton1966expected}. Here, sets of parallel lines are generalized to random sets of rectifiable curves in the plane.  Under some mild assumptions, Morton obtained (for details see \cite{solomon1978geometric}) a \textit{sine law} and showed that the density of the intersection angle between the tangents of any curve is 
\[
\frac{1}{2} \sin\lP\theta\rP\mathds{1}_{[0,\pi]}(\theta).
\]
Cai et al. \cite{cai2013distributions} studied the asymptotic behavior of pairwise angles among $n$ random and uniformly distributed unit vectors in $\R^p$ when $n$ goes to infinity and the dimension $p$ is either fixed or growing with $n$. They showed that the empirical distribution of the angles among the $n$ unit vectors converges weakly to the distribution with density 
\[
\frac{1}{\sqrt{\pi}} \frac{\Gamma\lP\frac{p}{2}\rP}{\Gamma\lP\frac{p-1}{2}\rP} \lP\sin(\theta)\rP^{p-2}\mathds{1}_{[0,\pi]}(\theta),
\] 
with probability one when $p$ is fixed.

Barton et al. \cite{10.2307/2333743} studied  chromosome interactions in human cells. In their model cell's nucleus is represented as a circle and  chromosomes as points randomly dropped inside that circle; see also Solomon \cite{solomon1978geometric}.  Interactions between chromosomes are determined by  intersections of random chords and the angles at which they intersect. The density in such models is always an expression involving the trigonometric \textit{sine function}.


Now we explain the main goal and contribution of the paper.  Our starting motivation is the afore mentioned model for wireless networks, for which explicit formulas for densities are non available, as we mentioned before. We propose a simple model in analogy with the random mechanism  in the wireless network model and for which we are able to give a closed form for the density function of a random angle. In this goal, we obtain a sine law identical to that of Morton's model.  Indeed, in spite of being different models for different purposes they intersect in their main object: a random angle. Our result contributes to Geometric Probability and its applications, and provides also another way to estimate $\pi$'s value.\\

After this introduction, the paper is organized as follows. In Section \ref{sec:mainresult}, we present our model and main Theorem \ref{thm:Main}. Section \ref{secProof} is dedicated to Theorem \ref{thm:Main}'s  proof. The proof itself is given in Subsection \ref{subsec:TheProof},  with the lengthy computation of a definite integral  deferred to Appendix \ref{appSecondIntegral}. The previous two subsections \ref{subs:case1} and \ref{subs:case2} introduce necessary elements and formulas used in the proof of Theorem \ref{thm:Main}. However, the propositions there are not directly linked. Section \ref{densityThetabyConvolution} concludes with a problem for future work.
\section{Main result}\label{sec:mainresult}
We denote by $\|\cdot\|$ the usual Euclidean norm in $\R^2$. Let $\ball=\lL x\in \R^2 \mid \| x\|\leq 1\rL$. We start by  introducing the model for a  random angle between two chords in the unitary disc. The model consists of the following objects:
\begin{itemize}
\item A chord  $s_1:=\overline{AB}$ determined by two independent points $A$ and $B$ with uniform distribution on the disc.  
\item An horizontal chord $s_2$ with random height $Y$ of uniform distribution on $[-1,1]$. We take $s_1$ independent from $s_2$.
\item $z$ is the intersecting point of the lines determined by $s_1$ and $s_2$.  Special interest is on  the event  where the intersection lies inside the disc:  
\begin{equation}\label{eq:EventC}
\eC:=\{z \in \ball\}.
\end{equation}
\item $\Theta$ is the angle between the lines determined by $s_1$ and $s_2$ measured counterclockwise.
\end{itemize}

The next theorem is the  main result of the paper. Here the exact density of $\Theta$ conditioned on $\eC$ is determined.
\begin{theorem}\label{thm:Main}
The conditional density of the angle  $\Theta$ conditional on $\eC$  is given by
\[
\frac{1}{2} \sin\lP\theta\rP \mathds{1}_{[0,\pi]}(\theta).
\]
In particular, $\Prob{\eC} = \frac{256}{45\pi^2}$.
\end{theorem}

The proof of Theorem \ref{thm:Main} is organized in Section \ref{secProof}.  The proof itself is given in Subsection \ref{subsec:TheProof},  with the computation of a crucial integral  deferred to Appendix \ref{appSecondIntegral}. The previous two subsections \ref{subs:case1}  and \ref{subs:case2} are preliminary and  introduce necessary elements and formulas used in the proof. The propositions in this section, being interesting in their own are not directly linked.

\section{Proof of Theorem \ref{thm:Main}}\label{secProof}
\subsection{The angle determined by  $s_1$}\label{subs:case1}
Let $\Theta_1$ be the angle between  the chord $s_1$ and the $x$-axis. In this subsection we will determine $\Theta_1$'s distribution.    Here, and only in this subsection, we are interested in the whole distribution of $\Theta_1$, without any conditioning.   Let $\xi$ be a random variable with uniform distribution on $\ball$. Observe that  $\xi$ can be parameterized as
\begin{equation}\label{eqRT}
\xi = (R\cos(T),R\sin(T)),
\end{equation}
where $R$ and $T$ are independent random variables with densities respectively given by 
\begin{align}
f(r)&=2r\mathds{1}_{r\in[0,1]},\notag\\ 
g(\beta)&=\frac{1}{2\pi}\mathds{1}_{\beta\in[0,2\pi]}.\label{eqRTdensities}
\end{align}
\begin{proposition}
The distribution of $\Theta_1$ is uniform on $[0,\pi]$.
\end{proposition}
\begin{proof} 
We will also use the alternative notation $\angle AB$ for the angle $\Theta_1$ in order to emphasize the role of the  points $A$ and $B$, since we will fix one of them and let the other to be variable.
Take $\theta\in[0,\pi]$ and let $\Delta\theta$ be a small non negative number. Using conditional probability and the independency between $R$ and $T$, we have
\begin{eqnarray}\label{eqn13082019}
\Prob{\angle AB\in[\theta,\theta+\Delta\theta]} & = & \int\PC{\angle AB\in[\theta,\theta+\Delta\theta]}{A=a}d\mu_A(a) \nonumber\\
& = & \int \Prob{\angle aB\in[\theta,\theta+\Delta\theta]}d\mu_A(a),
\end{eqnarray} where $\mu_A$ denotes the distribution associated to $A$; see \eqref{eqRT} and \eqref{eqRTdensities}.  Note that when $\Delta\theta\to 0$ we have
\begin{equation}\label{eqn130820191558}
\frac{1}{\Delta\theta}\Prob{\angle aB\in[\theta,\theta+\Delta\theta]} \to \textnormal{ desired density}.
\end{equation} 
Up to a first order error we have (see Figure \ref{fig301020193}) 
\begin{equation}\label{eq:CircularSectors}
\Prob{\angle aB\in[\theta,\theta+\Delta\theta]} = \dist{a}{V}^2 + \dist{a}{W}^2 
\end{equation}

\begin{figure}[h]
\centering
\captionsetup{justification=centering}
\scalebox{0.4}{\includegraphics{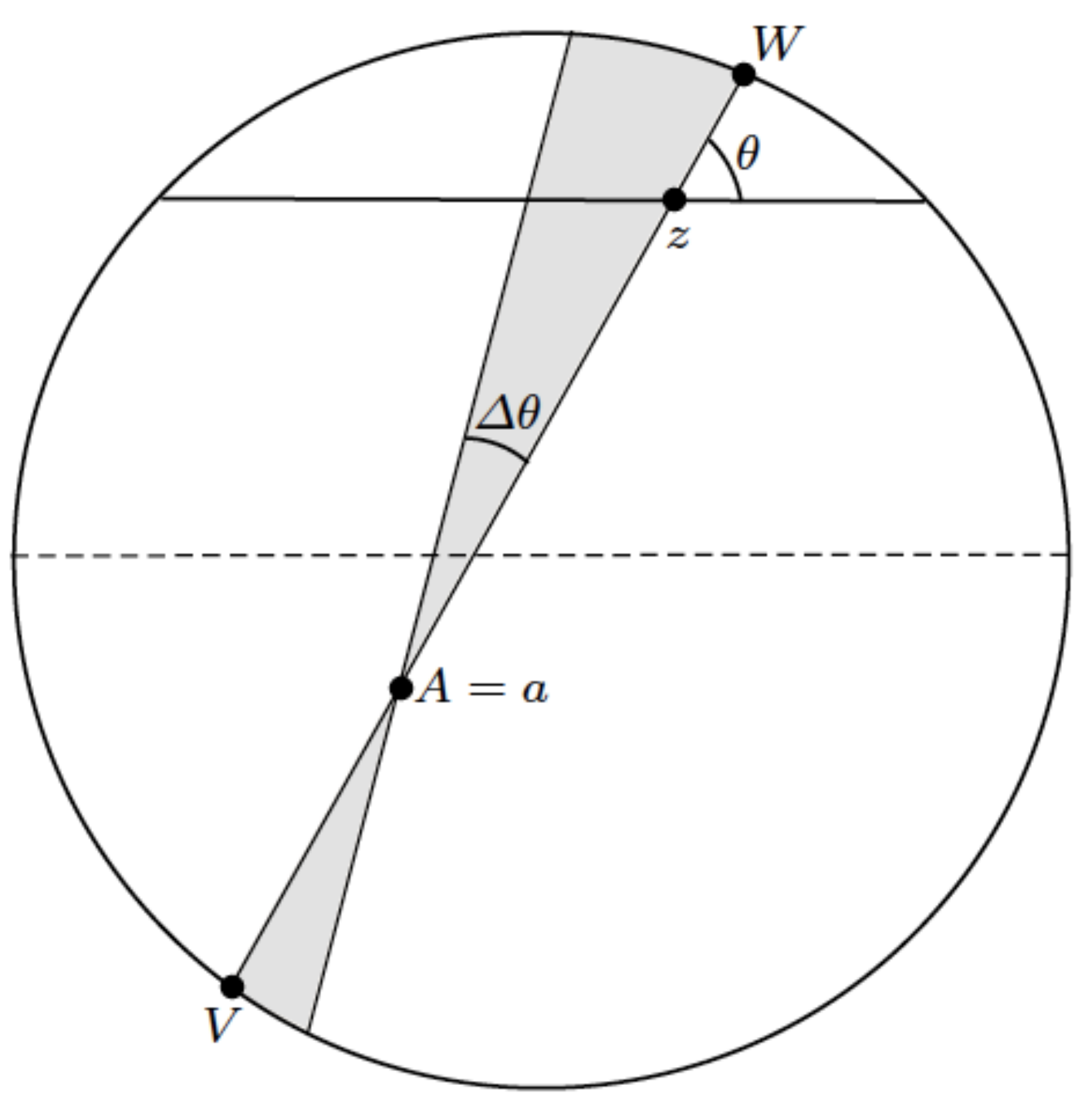}}
\caption{Points $B$ in the shadowed region determine $\Theta_1 \in [\theta,\theta + \Delta \theta]$.}\label{fig301020193}
\end{figure}

where $V,W\in\R^2$ are the solutions of the following system:
\begin{eqnarray}\label{eq:sys1}
x^2+y^2 & = & 1 \nonumber  \\
y   & = & m (x-a_1) + a_2,\\
m  &= & \tan(\theta) \nonumber
\end{eqnarray} 
with $a=(a_1,a_2)$. Explicit solutions are 
\begin{eqnarray*}
V_1 & = & \frac{m^2a_1-ma_2-\sqrt{1+m^2-\left(ma_1-a_2\right)^2}}{1+m^2}, \\
V_2 & = & -ma_1+a_2+\frac{m^3a_1}{1+m^2}-\frac{m^2a_2}{1+m^2} - \frac{m\sqrt{1+m^2-\left(ma_1-a_2\right)^2}}{1+m^2}, \\
W_1 & = & \frac{m^2a_1-ma_2+\sqrt{1+m^2-\left(ma_1-a_2\right)^2}}{1+m^2}, \\
W_2 & = & -ma_1+a_2+\frac{m^3a_1}{1+m^2}-\frac{m^2a_2}{1+m^2} + \frac{m\sqrt{1+m^2-\left(ma_1-a_2\right)^2}}{1+m^2}.
\end{eqnarray*} Also, by the system \eqref{eq:sys1}, we observe 
\begin{eqnarray*}
\dist{a}{V}^2 & = & (a_1-V_1)^2 + (a_2-V_2)^2 \\
& = &  (a_1-V_1)^2 + \left(a_2- \left(m(V_1-a_1)+a_2\right)\right)^2 \\
& = & \left(1+m^2 \right)(a_1-V_1)^2. \end{eqnarray*} Similarly, $\dist{a}{W}^2 = \left(1+m^2 \right)(a_1-W_1)^2$. Thus,
\begin{eqnarray*}
\dist{a}{V}^2 + \dist{a}{W}^2 & = & \left(1+m^2 \right)\left( (a_1-V_1)^2 + (a_1-W_1)^2\right)\\
& = & 2\frac{1+m^2+\left(1-m^2\right)\left(a_1^2- a_2^2\right) +4ma_1a_2}{1+m^2}.
\end{eqnarray*} 
The point $a$ is taken according to the law of $\xi$ so 
recalling the parameterization \eqref{eqRT},  we set $a=(r_A\cos(\beta_A),r_A\sin(\beta_A))$ and $m=\tan(\theta)$. Then
\begin{eqnarray}\label{eqn130820191256}
\dist{a}{V}^2 + \dist{a}{W}^2 & = & 2\left(1+r^2_A\cos\left(2\left( \theta - \beta_A \right)\right)\right). 
\end{eqnarray} From (\ref{eqn13082019}),  (\ref{eqn130820191558}) and (\ref{eqn130820191256}), we have 
\begin{eqnarray*}
&&\int \Prob{\angle aB\in[\theta,\theta+\Delta\theta]}  d\mu_A(a)  =  \int \frac{\Delta\theta}{2\pi}  \left(\dist{a}{V}^2 + \dist{a}{W}^2\right) d\mu_A(a)\\
 & &\hspace{1.5cm} = \frac{\Delta\theta}{\pi}  \int \left(1+r^2_A\cos\left(2\left( \theta - \beta_A \right)\right)\right) d\mu_A(a) \\
& & \hspace{1.5cm} = \frac{\Delta\theta}{\pi} \int_0^{2\pi}\int_0^1 \left(1+r^2_A\cos\left(2\left( \theta - \beta_A \right)\right)\right) \cdot \frac{1}{2\pi} \cdot 2r_A dr_Ad\beta_A \\
& & \hspace{1.5cm} = \frac{\Delta\theta}{\pi}.
\end{eqnarray*}
Thus, the distribution of the angle $\Theta_1$ is uniform on $[0,\pi]$ as claimed.
\end{proof}

\subsection{The angle of $s_1$ conditional on $\eC_1$}\label{subs:case2}
In this part the setting is similar to that of Subsection \ref{subs:case1}, but here we are interested in the distribution of $\Theta_1$ conditional to the event $\eC_1$ in which the intersection point $z$ of $s_1$ with the horizontal diameter on the $x$-axis  lies inside of the disc $\ball$. Thus, here we are interested in the following conditional probability 
\begin{equation}\label{eqn140820191416}
\PC{\Theta_1\in[\theta,\theta+\Delta\theta]}{\eC_1}.
\end{equation}

\begin{figure}[h]
\centering
\captionsetup{justification=centering}
\scalebox{0.5}{\includegraphics{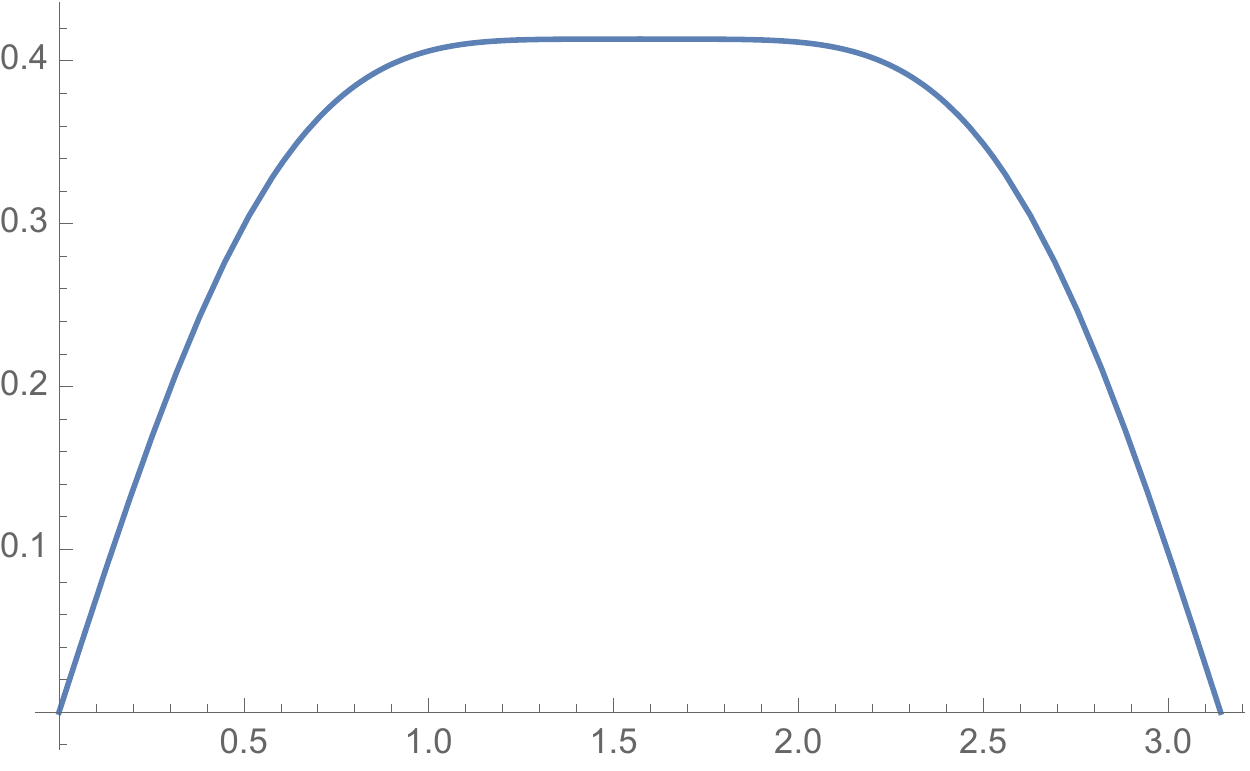}}
\caption{The density $g(\cdot)$ of $\Theta_1$ conditional on $\eC_1$.}\label{fig150820191405}
\end{figure}

In the next result, the conditional probability  \eqref{eqn140820191416} is completely characterized through a density function $g$  determined in closed form. The density function is illustrated in Figure \ref{fig150820191405}.
\begin{proposition}
Let $d(\theta) =8\sin(2\theta)+\sin(4\theta)$, $a=16+3\pi^2$  and  
\begin{equation}
g(\theta) =
\left\{
\begin{array}{ll}
\frac{1}{a} 12\theta + \frac{1}{a} d(\theta) & \mbox{\;if $\theta\in[0,\pi/2]$} \\[0.2cm]
\frac{1}{a} 12(\pi-\theta) - \frac{1}{a} d(\theta) & \mbox{\;if $\theta\in[\pi/2,\pi]$}
\end{array}.
\right. 
\end{equation}
Then, we have
\begin{equation}\label{eq:Dens2}
\PC{\Theta_1\in[\theta,\theta+\Delta\theta]}{\eC_1} =\int_{\theta}^{\theta + \Delta \theta} g(t) dt.
\end{equation}
In particular  $\Prob{\eC_1} = \frac{16+3\pi^2}{6\pi^2}$.
\end{proposition}
\begin{proof}
By the definition of conditional probability, we have
\[
\PC{\Theta_1\in[\theta,\theta+\Delta\theta]}{\eC_1} = \frac{\Prob{\{ \Theta_1\in[\theta,\theta+\Delta\theta]\} \cap \eC_1}}{\Prob{\eC_1}}.
\] 
We will focus on calculating
\[
\Prob{\Theta_1\in[\theta,\theta+\Delta\theta] , \eC_1}:=\Prob{\{ \Theta_1\in[\theta,\theta+\Delta\theta]\} \cap \eC_1},
\] 
and then, $\Prob{\eC_1}$ will be determined together with the desired conditional probability and its density.  Note that \eqref{eqn130820191256} is equal to $2\left(1+r^2_A-2r^2_A\sin^2\left(\theta-\beta_A\right)\right)$. The same logic as in \eqref{eqn13082019}, \eqref{eqn130820191558}, and \eqref{eqn130820191256} yields 
\begin{eqnarray} \label{eqn140820191443}
& & \Prob{\Theta_1\in[\theta,\theta+\Delta\theta],\eC_1}  =  \int \PC{\Theta_1^A\in[\theta,\theta+\Delta\theta],\eC_1}{A=a} d\mu_A(a)   \nonumber\\
& & \hspace{2cm}  = \int \Prob{\Theta_1^a \in[\theta,\theta+\Delta\theta],\eC_1} d\mu_A(a) \nonumber\\
& & \hspace{2cm}  = \int \frac{\Delta\theta}{2\pi}\left(\dist{a}{V}^2 + \dist{a}{W}^2\right) \mathds{1}_{\eC_1} d\mu_A(a) \nonumber\\
& & \hspace{2cm}  = \frac{\Delta\theta}{\pi} \int \left(1+r^2_A-2r^2_A\sin^2\left(\theta-\beta_A\right)\right)  \mathds{1}_{\eC_1} d\mu_A(a). \nonumber \\
& &
\end{eqnarray} 
Note  that $z$ is the solution of the following system
\begin{eqnarray}\label{eq:sys2}
y & = & 0  \nonumber \\
y   & = & m (x-a_1) + a_2\\
m  & = & \tan(\theta). \nonumber
\end{eqnarray} 
The unique solution is given by  $z=\left( a_1 - \frac{a_2}{m}, 0\right)$. Using again the parameterization $a=(r_A\cos(\beta_A),r_A\sin(\beta_A))$ we get
\begin{equation}\label{eqn140820191502}
\|z\|^2 = \left(r_A \cos(\beta_A)-\frac{r_A\sin(\beta_A)}{\tan(\theta)}\right)^2 
= \frac{r^2_A\sin^2\left(\theta -\beta_A\right)}{\sin^2\left( \theta \right)}.
\end{equation} 
Using the linearity of integral, we see that the last integral in (\ref{eqn140820191443}) can be calculated in three parts as follows:
\begin{eqnarray}\label{eqn140820191729}
&& \int \left(1+r^2_A-2r^2_A\sin^2\left(\theta-\beta_A\right)\right)  \mathds{1}_{\{\|z_a\|^2\leq 1\}} d\mu_A(a) = \nonumber\\
&&\hspace{2cm} = \int \mathds{1}_{\{ \|z_a\|^2\leq 1\}} d\mu_A(a)  +  \int r_A^2 \mathds{1}_{\{\|z_a\|^2\leq 1\}} d\mu_A(a) \nonumber \\
&&\hspace{2.3cm} -2  \int r_A^2\sin^2 \left(\theta - \beta_A\right) \mathds{1}_{\{\|z_a\|^2\leq 1\}} d\mu_A(a). 
\end{eqnarray} 
For the first integral in the right hand side of \eqref{eqn140820191729}, using the expression \eqref{eqn140820191502} we have
\[
\int \mathds{1}_{\{\|z_a\|^2\leq 1\}} d\mu_A(a) = \int \mathds{1}_{\{\abs{t}\leq \sin(\theta)\}} d\mu_\eta(t),
\] 
where $\mu_\eta$ is the distribution associated to the random variable \begin{equation}\label{eq:eta}
\eta:=R\sin\left(\theta - T\right).
\end{equation}
Hence, if we calculate the law of $\eta$, the previous integral is straightforward. It is well known that if $T$ is uniform on $[0,2\pi]$, then the density of $\sin\left(\theta - T\right)$ is the arc sine law for all $\theta$, i.e., its density is
\begin{equation*}
h(t)=\frac{1}{\pi\sqrt{1-t^2}}\mathds{1}_{\{ \abs{t}<1 \}}.
\end{equation*}
Note that $R$ and $\sin(\theta - T)$ are independent random variables. Thus, the density of $\eta$ can be calculated using the convolution formula
\[ 
f_\eta(t)=\int_{-\infty}^{\infty} f(u)h\left( \frac{t}{u}\right)\frac{1}{\abs{u}} du.
\]
In our case,we obtain
\begin{equation}\label{eqn150820191214}
f_\eta(t)=\frac{2}{\pi}\sqrt{1-t^2}\mathds{1}_{\{ \abs{t}<1 \}}.
\end{equation} 
Thus, 
\begin{eqnarray}
\int \mathds{1}_{\{\abs{t}\leq \sin(\theta)\}} d\mu_\eta(t) & = & \int_{-\sin(\theta)}^{\sin(
\theta)} \frac{2}{\pi}\sqrt{1-t^2} dt \nonumber \\
& = & \frac{2}{\pi} \left(\arcsin(\sin(\theta)) + \sin(\theta)\sqrt{1-\sin^2(\theta)}\right). 
\end{eqnarray}
For the second integral in the right hand side of \eqref{eqn140820191729}, we consider two parts. For the first, take $\theta\in[0,\pi/2]$.  We make use of the elementary inequalities:
\begin{eqnarray}\label{eq:Elementary1}
\left| \sin(x) \right| \geq  \left| \sin(\theta) \right|   & \mbox{for}  & x \in [-\pi+\theta,-\theta] \cup [\theta,\pi-\theta] \cup [\pi + \theta,2\pi-\theta] \nonumber\\ [1mm]
\left| \sin(x) \right| \leq  \left| \sin(\theta) \right|   &\mbox{for}  & x \in [-\pi,-\pi+\theta] \cup [-\theta,\theta] \cup [\pi - \theta,\pi+\theta].
\end{eqnarray} 
Writting $r_A^2=\gamma$ we have
\begin{eqnarray*}
&& \int r_A^2 \mathds{1}_{\{ \|z_a\|^2\leq 1\}} d\mu_A(a) = \int \int r^2_A \mathds{1}_{ \{r^2_A\sin^2(\theta-\beta_A)\leq \sin^2(\theta)\}}  d\mu_{R_A}(r_A) d\mu_{T}(\beta_A) \\
&&\hspace{1cm} = \int_0^{2\pi} \int_0^1 \gamma \mathds{1}_{\{\gamma\sin^2(\theta-\beta_A)\leq \sin^2(\theta)\}}  \frac{1}{2\pi}d\gamma d\beta_A \\
&&\hspace{1cm} = \int_0^{2\theta} \left[\int_0^1 \frac{\gamma}{2\pi} d\gamma\right] d\beta_A + \int_{2\theta}^{\pi} \left[\int_0^{\frac{\sin^2(\theta)}{\sin^2(\theta-\beta_A)}} \frac{\gamma}{2\pi} d\gamma\right] d\beta_A \\
&&\hspace{1.5cm} + \int_{\pi}^{2\theta+\pi} \left[\int_0^1 \frac{\gamma}{2\pi} d\gamma\right] d\beta_A + \int_{2\theta+\pi}^{2\pi} \left[\int_0^{\frac{\sin^2(\theta)}{\sin^2(\theta-\beta_A)}} \frac{\gamma}{2\pi} d\gamma\right] d\beta_A \\
&&\hspace{1cm} = \frac{\theta}{\pi} + \frac{\left(2-\cos(2\theta)\right)\sin(2\theta)}{6\pi}.
\end{eqnarray*}
For the second, take $\theta\in[\pi/2,\pi]$. The analogous inequalities to  \eqref{eq:Elementary1} in the new region are
\begin{eqnarray}\label{eq:Elementary2}
\left| \sin(x) \right| \geq  \left| \sin(\theta) \right|   & \mbox{for} & x \in [-\theta,-\pi+\theta] \cup [\pi-\theta,\theta] \cup [2\pi-\theta,\pi + \theta], \nonumber\\ [1mm]
\left| \sin(x) \right| \leq  \left| \sin(\theta) \right|   & \mbox{for} & x \in [-\pi,-\theta] \cup [-\pi+\theta,\pi-\theta] \cup [\theta,2\pi-\theta]\cup [\pi + \theta,2\pi]. \nonumber\\
\end{eqnarray}
Now we have
\begin{eqnarray*}
&& \int r_A^2 \mathds{1}_{\{\|z_a\|^2\leq 1\}} d\mu_A(a)  = \int_0^{2\pi} \int_0^1 \gamma \mathds{1}_{\{\gamma\sin^2(\theta-\beta_A)\leq \sin^2(\theta)\}}  \frac{1}{2\pi}d\gamma d\beta_A \\
&&\hspace{1cm} = \int_0^{2\theta-\pi}\left[\int_0^{\frac{\sin^2(\theta)}{\sin^2(\theta-\beta_A)}} \frac{\gamma}{2\pi} d\gamma\right] d\beta_A  + \int_{2\theta-\pi}^{\pi} \left[\int_0^{1} \frac{\gamma}{2\pi} d\gamma\right] d\beta_A \\
&&\hspace{1.5cm} + \int_{\pi}^{2\theta} \left[\int_0^{\frac{\sin^2(\theta)}{\sin^2(\theta-\beta_A)}} \frac{\gamma}{2\pi} d\gamma\right] d\beta_A  + \int_{2\theta}^{2\pi} \left[\int_0^{1} \frac{\gamma}{2\pi} d\gamma\right] d\beta_A \\
&&\hspace{1cm} = \frac{\pi-\theta}{\pi}-\frac{\left(2-\cos(2\theta)\right)\sin(2\theta)}{6\pi}.
\end{eqnarray*} 
Thus, $ \int r_A^2 \mathds{1}_{\{\|z_a\|^2\leq 1\}} d\mu_A(a) = H(\theta)$, where $H(\theta)$ is
$$ \label{eqn150820190947}
H(\theta) = \left\{
\begin{array}{ll}
\frac{\theta}{\pi} + \frac{\left(2-\cos(2\theta)\right)\sin(2\theta)}{6\pi} & \mbox{\;if $\theta\in[0,\pi/2]$} \\[0.2cm]
\frac{\pi-\theta}{\pi}-\frac{\left(2-\cos(2\theta)\right)\sin(2\theta)}{6\pi} & \mbox{\;if $\theta\in[\pi/2,\pi]$}
\end{array}.
\right.
$$
Taking $\eta=r_A\sin(\theta-\beta_A)$. From (\ref{eqn150820191214}), we obtain for the third integral in the right hand side of (\ref{eqn140820191729}) that
\begin{eqnarray*}
& &-2  \int r_A^2\sin^2 \left(\theta - \beta_A\right) \mathds{1}_{\{ \|z_a\|^2\leq 1 \}} d\mu_A(a) = -2 \int t^2 \mathds{1}_{\{ \abs{t}\leq \sin(\theta)\} } d\mu_\eta(t) \\
& & \hspace{1cm} = -\frac{4}{\pi}\int_{-\sin(\theta)}^{\sin(\theta)} t^2 \sqrt{1-t^2} dt
= -\frac{8}{\pi} \int_{0}^{\sin(\theta)} t^2 \sqrt{1-t^2} dt\\
& & \hspace{1cm} = -\frac{1}{\pi}\left( \arcsin(\sin(\theta))-\abs{\cos(\theta)}\cos(2\theta)\sin(\theta) \right). 
\end{eqnarray*}
Now equation \eqref{eq:Dens2} follows from an easy normalization.
\end{proof}

\subsection{Horizontal chord with random height}
\label{subsec:TheProof}
In this section we conclude the proof of the main Theorem  \ref{thm:Main}. We will make use of some formulas of previous subsections, although not directly the propositions there.\\

Recall $Y$ is a random variable with uniform distribution on $[-1,1]$ representing the height of $s_2$ and independent of $s_1$. The random variable $Y$ can be expressed as $Y=\sin(\vG)$, where $\vG$ is a random variable with distribution, resp., density
\begin{equation}\label{eqn20270820191637}
\mu_{\vG}(d \rho)=f_{\vG}(\rho) d \rho=\frac{1}{2}\cos(\rho)\mathds{1}_{\left[-\frac{\pi}{2},\frac{\pi}{2}\right]}(\rho) d \rho.
\end{equation}

Recall that our goal is to explicitly calculate the density of the angle $\Theta$ between $s_1$ and $s_2$ conditional to the event in which the intersecting point $z$  lies inside of $\ball$:
\[
\PC{\Theta\in[\theta,\theta+\Delta\theta]}{\eC}.
\]
Similar to previous subsections, it will be sufficient to focus on  
\begin{equation*}
\Prob{\Theta\in[\theta,\theta+\Delta\theta], \eC}:=\Prob{\{ \Theta\in[\theta,\theta+\Delta\theta] \} \cap  \eC}.
\end{equation*}

We continue denoting by $V$ and $W$ the solutions of the system \eqref{eq:sys1}. The coordinates of the intersection point $z$ is a solution to the following system (compare with \eqref{eq:sys2})
\begin{eqnarray}\label{eq:sys3}
y   &  = & \sin(\rho) \nonumber\\
y   & =  &m (x-a_1) + a_2\\
m  & =  &\tan(\theta). \nonumber
\end{eqnarray}
 
The point $z$ is then parameterized by $\theta$, $\rho$ and $a$. The point $a$ is parameterized by $(r_A\cos(\beta_A),r_A\sin(\beta_A))$, see \eqref{eqRT} and \eqref{eqRTdensities}, so will be $z$.  Hence, the norm of $z$ (as a random variable) is explicitely given by
\begin{eqnarray}\label{eqn270820190944}
 \|z\|^2 & =  & \sin^2(\vG) + \frac{\left(\cos(\theta)\sin(\vG) +R\sin(\theta - T)\right)^2}{\sin^2(\theta)}.
\end{eqnarray}
Note that $\Theta$ is analogously parameterized and then
\begin{eqnarray}\label{eqn260820191911}
& & \Prob{\Theta\in[\theta,\theta+\Delta\theta], \eC} \nonumber\\
& & \hspace{0.5cm} =  \int \int \Prob{\Theta_{\rho, a}\in[\theta,\theta+\Delta\theta]} d\mu_{A}(a) d\mu_\vG(\rho) \nonumber \\
& &  \hspace{0.5cm} = \frac{\Delta\theta}{\pi}\left[ \int \int  \mathds{1}_{\{z_{\rho,a} \in \ball\}}   d\mu_{A}(a) d\mu_\vG(\rho) + \int \int  r^2_A\mathds{1}_{\{z_{\rho,a} \in \ball\}}   d\mu_{A}(a) d\mu_\vG(\rho)  \right. \nonumber\\
& &\hspace{1cm} \left. -\int \int  2r^2_A\sin^2\left(\theta-\beta_A\right)\mathds{1}_{\{z_{\rho,a} \in \ball\}}   d\mu_{A}(a) d\mu_\vG(\rho) \right]. 
\end{eqnarray}  
\begin{remark}
Now the proof of Theorem \ref{thm:Main} will be established after we compute explicitly each integral in the last equality  in \eqref{eqn260820191911}.  The first and third are simple. Quite surprisingly,  the second integral collapses to a very simple expression after a lengthy and complex computation resulting from the extremely complicated structure of the integration domain  determined by the event $\eC$.
\end{remark}

\begin{remark}\label{rem05102019} Note that each integral in the expression \eqref{eqn260820191911} is as a function of $\theta$ symmetric at $\pi/2$. Hence, it is sufficient to determinate the density function for $\theta\in[0,\pi/2]$.
\end{remark}

In the setting of the present subsection, after \eqref{eqn270820190944} and  \eqref{eq:eta}, the set $\eC$ takes the form 
\begin{equation}\label{eqn270820191632}
\eC=\left\{-\sin(\theta + \rho) \leq \eta \leq \sin(\theta - \rho)\right\},
\end{equation} with $\eta$ a random variable with density \eqref{eqn150820191214}. 
Note that \eqref{eqn270820191632} must be  a subset of 
\begin{align*}
\{(\theta,\rho) & \in [0,\pi]\times[-\pi/2,\pi/2]  \mid -\sin(\theta + \rho) \leq \sin(\theta - \rho)
\} \\
&=\{(\theta,\rho) \mid \sin(\theta) \cos(\rho)\geq 0\} =[0,\pi]\times[-\pi/2,\pi/2].
\end{align*}

\subsubsection{First integral}
\begin{lemma}
We have
\begin{equation*}
\int \int  \mathds{1}_{\{ \|z_{\rho,a}\|_2^2\leq 1 \} }   d\mu_{A}(a) d\mu_\vG(\rho) = \frac{8\sin(\theta)}{3\pi}.
\end{equation*}
\end{lemma}
\begin{proof}
Using the expressions \eqref{eqn150820191214}, \eqref{eqn20270820191637}, and \eqref{eqn270820191632}, the first integral of the right hand side of \eqref{eqn260820191911} when $\theta\in[0,\pi/4]$ is
\begin{eqnarray}\label{eqn280820191300}
& & \int \int  \mathds{1}_{\{ \|z_{\rho,a}\|_2^2\leq 1 \}}   d\mu_{A}(a) d\mu_\vG(\rho) =  \int_{-\pi/2}^{\pi/2} \left[ \int_{-\sin(\theta+\rho)}^{\sin(\theta-\rho)} \frac{2}{\pi} \sqrt{1-t^2} dt\right] \frac{1}{2}\cos(\rho) d\rho \nonumber \\
& & = \frac{1}{2\pi}\int_{-\pi/2}^{-\pi/2+\theta} \left[-\frac{1}{2}\sin(2(\theta - \rho)) + \frac{1}{2}\sin(2(\theta + \rho)) + \pi + 2\rho \right] \cos(\rho) d\rho \nonumber \\
& &\hspace{0.5cm} + \frac{1}{2\pi} \int_{-\pi/2+\theta}^{\pi/2-\theta} \left[\frac{1}{2}\sin(2(\theta - \rho)) + \frac{1}{2}\sin(2(\theta + \rho)) + 2\theta \right]  \cos(\rho) d\rho \nonumber \\
& &\hspace{0.5cm} + \frac{1}{2\pi} \int_{\pi/2-\theta}^{\pi/2} \left[\frac{1}{2}\sin(2(\theta - \rho)) - \frac{1}{2}\sin(2(\theta + \rho)) + \pi - 2\rho \right]  \cos(\rho) d\rho \nonumber \\
& & = \frac{8\sin(\theta)}{3\pi}. \nonumber
\end{eqnarray}

When $\theta\in[\pi/4,\pi/2]$, we have 
\begin{eqnarray}\label{eqn140920190907}
& & \int \int  \mathds{1}_{\{ \|z_{\rho,a}\|_2^2\leq 1\}}   d\mu_{A}(a) d\mu_\vG(\rho) =  \int_{-\pi/2}^{\pi/2} \left[ \int_{-\sin(\theta+\rho)}^{\sin(\theta-\rho)} \frac{2}{\pi} \sqrt{1-t^2} dt\right] \frac{1}{2}\cos(\rho) d\rho \nonumber \\
& & = \frac{1}{2\pi}\int_{-\pi/2}^{\theta-\pi/2} \left[-\frac{1}{2}\sin(2(\theta - \rho)) + \frac{1}{2}\sin(2(\theta + \rho)) + \pi + 2\rho \right] \cos(\rho) d\rho \nonumber \\
& &\hspace{0.5cm} + \frac{1}{2\pi} \int_{\theta-\pi/2}^{\pi/2-\theta} \left[\frac{1}{2}\sin(2(\theta - \rho)) + \frac{1}{2}\sin(2(\theta + \rho)) + 2\theta \right]  \cos(\rho) d\rho \nonumber \\
& &\hspace{0.5cm} + \frac{1}{2\pi} \int_{\pi/2-\theta}^{\pi/2} \left[\frac{1}{2}\sin(2(\theta - \rho)) - \frac{1}{2}\sin(2(\theta + \rho)) + \pi - 2\rho \right]  \cos(\rho) d\rho \nonumber \\
& & =  \frac{8\sin(\theta)}{3\pi}. \nonumber 
\end{eqnarray}

Thus, for all $\theta\in[0,\pi]$
\begin{equation}
\int \int  \mathds{1}_{\{ \|z_{\rho,a}\|_2^2\leq 1 \}}   d\mu_{A}(a) d\mu_\vG(\rho) = \frac{8\sin(\theta)}{3\pi}.
\end{equation} 
\end{proof}
\subsubsection{Third integral}
\begin{lemma}
For $\theta\in[0,\pi]$  we have
\begin{equation}\label{eqn300820191123}
2\int \int  r^2_A\sin^2\left(\theta-\beta_A\right)\mathds{1}_{\mathcal{C}}   d\mu_{A}(a) d\mu_\vG(\rho) = \frac{16 \sin(\theta)}{15\pi}.
\end{equation}
\end{lemma}
\begin{proof}
Note that
\begin{eqnarray*}
& & 2\int \int  r^2_A\sin^2\left(\theta-\beta_A\right)\mathds{1}_{\mathcal{C}}   d\mu_{A}(a) d\mu_\vG(\rho) = 2\int_{-\pi/2}^{\pi/2} \left[ \int_{-\sin(\theta+\rho)}^{\sin(\theta-\rho)} \frac{2}{\pi}t^2 \sqrt{1-t^2} dt\right] \frac{1}{2}\cos(\rho)d\rho.
\end{eqnarray*} When $\theta\in[0,\pi/4]$, we have 
\begin{eqnarray} 
&  & \frac{2}{\pi} \int_{-\pi/2}^{\pi/2} \left[ \int_{-\sin(\theta+\rho)}^{\sin(\theta-\rho)} t^2 \sqrt{1-t^2} dt\right] \cos(\rho)d\rho \nonumber \\ 
& & = \frac{1}{4\pi} \int_{-\pi/2}^{-\pi/2+\theta}  \left[\frac{1}{4}\sin(4(\theta - \rho)) - \frac{1}{4}\sin(4(\theta + \rho)) + \pi + 2\rho \right] \cos(\rho) d\rho\nonumber \\
& & \hspace{0.5cm} + \frac{1}{4\pi} \int_{-\pi/2+\theta}^{\pi/2-\theta}  \left[-\frac{1}{4}\sin(4(\theta - \rho)) - \frac{1}{4}\sin(4(\theta + \rho)) + 2\theta \right] \cos(\rho) d\rho\nonumber \\
& &\hspace{0.5cm}  +\frac{1}{4\pi} \int_{\pi/2 - \theta}^{\pi/2}  \left[-\frac{1}{4}\sin(4(\theta - \rho)) + \frac{1}{4}\sin(4(\theta + \rho)) + \pi - 2\rho \right] \cos(\rho) d\rho\nonumber \\
& & = \frac{16\sin(\theta)}{15\pi}. \nonumber
\end{eqnarray}

When $\theta\in[\pi/4,\pi/2]$, we have 
\begin{eqnarray}\label{eqn150920191232}
& &  \frac{2}{\pi} \int_{-\pi/2}^{\pi/2} \left[ \int_{-\sin(\theta+\rho)}^{\sin(\theta-\rho)} t^2 \sqrt{1-t^2} dt\right] \cos(\rho)d\rho \nonumber \\ 
& & =  \frac{1}{4\pi}\int_{-\pi/2}^{\theta-\pi/2} \left[\frac{1}{4}\sin(4(\theta - \rho)) - \frac{1}{4}\sin(4(\theta + \rho)) + \pi + 2\rho \right] \cos(\rho) d\rho \nonumber \\
& &\hspace{0.5cm} + \frac{1}{4\pi} \int_{\theta-\pi/2}^{\pi/2-\theta} \left[-\frac{1}{4}\sin(4(\theta - \rho)) - \frac{1}{4}\sin(4(\theta + \rho)) + 2\theta \right]  \cos(\rho) d\rho \nonumber \\
& &\hspace{0.5cm} + \frac{1}{4\pi} \int_{\pi/2-\theta}^{\pi/2} \left[-\frac{1}{4}\sin(4(\theta - \rho)) + \frac{1}{4}\sin(4(\theta + \rho)) + \pi - 2\rho \right]  \cos(\rho) d\rho \nonumber \\
& & =  \frac{16\sin(\theta)}{15\pi}. \nonumber
\end{eqnarray}

Thus, for all $\theta\in[0,\pi]$
\begin{equation}\label{eqn300820191123}
2\int \int  r^2_A\sin^2\left(\theta-\beta_A\right)\mathds{1}_{\{\|z_{\rho,a}\|_2^2\leq 1\}}   d\mu_{A}(a) d\mu_\vG(\rho) = \frac{16\sin(\theta)}{15\pi}.
\end{equation}
\end{proof}

\subsubsection{Second integral}

The second integral is the most complicated of the integrals to be calculated. The difficulty comes from the structure of the event $\eC$, see \eqref{eqn270820191632}. It is necessary  to distinguish for every $\theta\in[0,\pi/2]$, every $\rho\in[-\pi/2,\pi/2]$, and every $\beta\in[0,2\pi]$  the interval where $r$ has a meaning. Thus, a careful analysis of integration intervals is required. This is done in Appendix \ref{appSecondIntegral} and as a clear consequence  we obtain for the second integral:
\begin{eqnarray}
& & \int \int  r^2_A\mathds{1}_{\{ \|z_{\rho,a}\|_2^2\leq 1 \}}   d\mu_{A}(a) d\mu_\vG(\rho) \nonumber \\
& & = \frac{1}{2\pi} \int_{-\pi/2}^{\pi/2} \left[ \int_{0}^{2\pi} \int_{0}^{1} r_A^3 \mathds{1}_{\mathcal{C}} dr_A d\beta_A\right] \cos(\rho) d\rho \nonumber\\
& & = \frac{56}{45\pi}\sin(\theta).
\end{eqnarray}

\section{Future work}\label{densityThetabyConvolution}
A natural extension to the present work is to consider segments instead of chords. It is interesting from the pure mathematical point of view and well motivated by engineering applications.

In this concluding section, we characterize the distribution of the angle with this random mechanism through  the convolution of two more simple distributions.  As the reader will witness,  it is unclear if this technique will lead to an explicit formula for the law, similar to  the statement in Theorem \ref{thm:Main}. Hence, we let for future work the study of a geometric approach for this model.\\

Consider points $A, B, C$ and $D$ independent and uniformly distributed on the disc $\ball$. Let $\phi_A, \phi_C$ be the random angles of the equivalent vectors of the points $A, C$, which are measured from the positive $x$-axis in counterclockwise direction. Note that $\phi_A, \phi_C$ are independent random variables with the same continuous uniform distribution on $[0,2\pi]$.

Let $\gamma_A$ be the random angle between the vector induced by the point $A$ and  the segment $s_1$, which is measured from the vector $A$ to the segment $s_1$ in counterclockwise direction. Similarly, we define $\gamma_C$ for the point $C$ and segment $s_2$; see Figure \ref{fig301020192}. We emphasize that the angle $\gamma_C$ is measured from the vector $C$ to the segment $s_2$ in counterclockwise direction. Note that  $\gamma_A$ and $\gamma_C$ are iid random variables whith distribution supported in $[0,2\pi]$. The random variables $\phi_A, \phi_C, \gamma_A, \gamma_C$ are independent and have different distributions. One might guess a continuous uniform distribution for $\gamma_A$, but this is not the case. The density  $f_{\gamma_A}$ of $\gamma_A$   is
\begin{eqnarray} \label{eqn150520191652}
f_{\gamma_A}(x)  & = & \frac{1}{4\pi \abs{\sin^3 x}} \lC \abs{\sin x} \lP -2\cos^4 x -2\cos^3 x\abs{\cos x} \right.\right. \nonumber \\
& & \left.\left.+ \cos^2 x + \cos x \abs{\cos x} + 1 \rP + \arcsin\lP \abs{\sin x}\rP\cos x\rC.
\end{eqnarray} For details of deduction of $f_{\gamma_A}$ see section 5 in \cite{Bettstetter2004}.

We introduce the following notation. For $a,b\in\R$ we write $a \equiv b\;\textnormal{mod}\; \pi$ if there exists $k\in\Z$ such that $a-b=k\pi$.
\begin{proposition}
The random angle between segments in the circle  has the same law as the random variable
\begin{equation}\label{eqn200502191853}
\lP \phi_A - \phi_C \rP + \lP \gamma_A - \gamma_C\rP \;\textnormal{mod}\; \pi.
\end{equation} 
\end{proposition}
\begin{proof}
Let $z$ be the intersection point determined by $s_1$ and $s_2$. Let $O$ be the origin in $\R^2$.  Consider the quadrilateral $Q$ induced by the points $O,A,B,z$. The inner angle in the vertex $O$ of $Q$ is the difference between $\phi_A$ and $\phi_C$. Then, 
\[
\angle O = \phi_A - \phi_C + k_1 \pi, 
\] 
where $k_1\in\Z$ is a random variable which depends on $\phi_A,\phi_C$. Moreover, $\phi_A - \phi_C$ is a symmetric random variable, i.e., $\phi_A - \phi_C \stackrel{D}{=}\phi_C - \phi_A$, where $\stackrel{D}{=}$ means ``equal in distribution''.

If the quadrilateral $Q$ is simple (see Figure \ref{fig301020192}, panel (a)), by our convention on the direction of $\gamma_A$ and $\gamma_C$, the sum of inner angle in the vertices $A,C$ is
$$
\angle A + \angle C = \gamma_A - \gamma_C + k_2\pi,
$$ where $k_2\in\Z$ is a random variable which depends on $\gamma_A,\gamma_C$. Note   that  $\gamma_A - \gamma_C$ is also a symmetric random variable. Hence, the inner angle in the vertex $z$ of $Q$ is
\begin{eqnarray}\label{eqn200520191839}
\angle z & = & 2\pi - \angle O -  \angle A -  \angle C \nonumber \\ 
& = & \lP \phi_A - \phi_C \rP + \lP \gamma_A - \gamma_C\rP + k_3\pi,
\end{eqnarray} where $k_3\in\Z$ is a random variable which depends on $\phi_A,\phi_C, \gamma_A,\gamma_C$.

\begin{figure}[h]
\centering
\captionsetup{justification=centering}
\scalebox{0.4}{\includegraphics{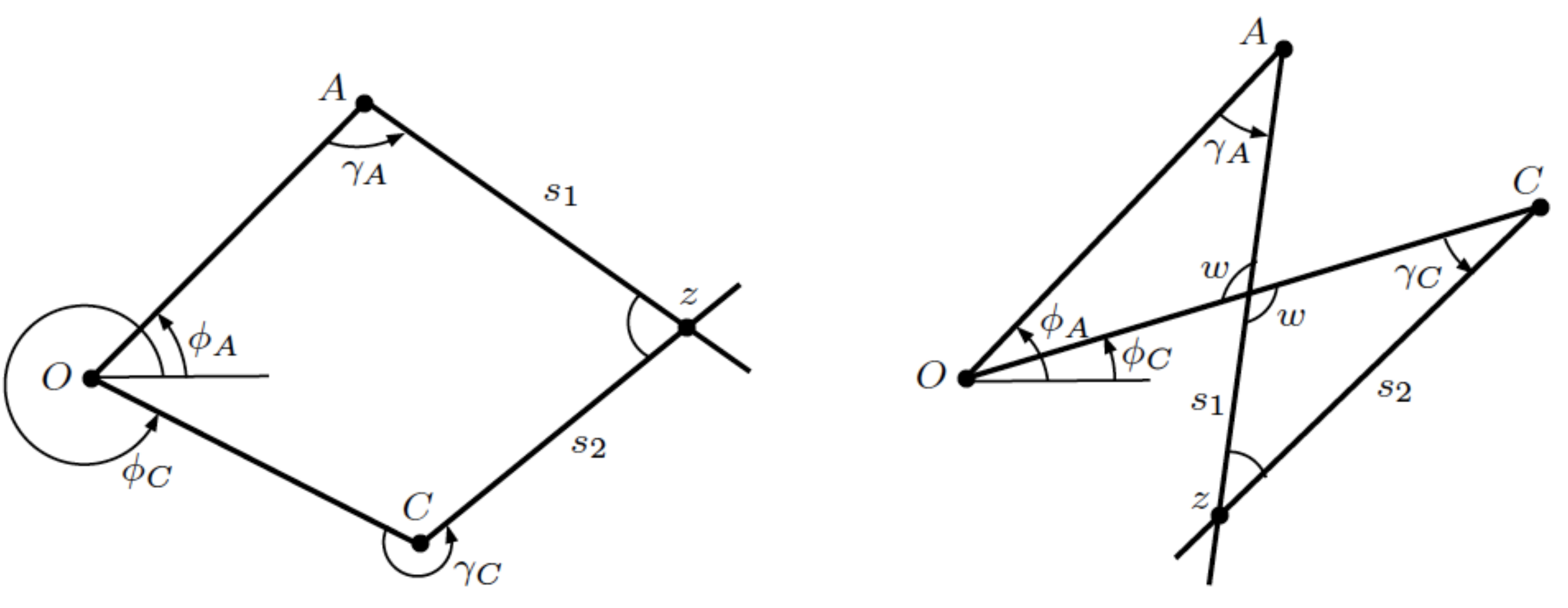}}
\caption{(a) A simple quadrilateral (b) A complex quadrilateral.}\label{fig301020192}
\end{figure}

In the case of a complex quadrilateral (see Figure \ref{fig301020192}, panel (b)), we  obtain a similar expression as in \eqref{eqn200520191839}.  Hence,
\[
\angle z \equiv \lP \phi_A - \phi_C \rP + \lP \gamma_A - \gamma_C\rP \;\textnormal{mod}\; \pi.
\]

Note that $\Theta \equiv \angle z\;\textnormal{mod}\; \pi$. Thus, \eqref{eqn200502191853} follows as claimed.
\end{proof}

\appendix
\section{Partition of Second Integral}\label{appSecondIntegral}
We use the notation $s_{\theta-\rho}:= \sin(\theta-\rho)$, $s_{\theta+\rho}:=\sin(\theta+\rho)$, and $s_{\theta-\beta_A}:=\sin(\theta-\beta_A)$. For the second integral in the right hand side of (\ref{eqn260820191911}). As mentioned in  Remark \ref{rem05102019} it is sufficient to take  $\theta\in[0,\pi/2]$. We have
\begin{equation*}
\int_{-\pi/2}^{\pi/2} \left[ \int_{0}^{2\pi} \int_{0}^{1} r_A^3 \mathds{1}_{\mathcal{C}} dr_A d\beta_A\right] \cos(\rho) d\rho = 
\left\{
\begin{array}{ll}
M_1+M_2+M_3+M_4 & \mbox{\;\; $\theta\in[0,\pi/6]$} \\ [1mm]
I_1+I_2+I_3+I_4 & \mbox{\;\; $\theta\in[\pi/6,\pi/4]$} \\ [1mm]
J_1+J_2+J_3+J_4 & \mbox{\;\; $\theta\in[\pi/4,\pi/3]$} \\ [1mm]
L_1+L_2+L_3+L_4 & \mbox{\;\; $\theta\in[\pi/3,\pi/2]$} 
\end{array},
\right.
\end{equation*} where:\\

\newpage
For $\theta\in[0,\pi/6]$

\begin{eqnarray}
& & M_1 = \int_0^\theta \left[ \integral{0}{\rho}{0}{\frac{s_{\theta-\rho}}{s_{\theta-\beta_A}}} + \integral{\rho}{2\theta+\rho}{0}{1} \right. \nonumber\\
& & \hspace{1.5cm} \left. + \integral{2\theta+\rho}{\pi-\rho}{0}{-\frac{s_{\theta+\rho}}{s_{\theta-\beta_A}}} + \integral{\pi-\rho}{\pi+2\theta-\rho}{0}{1} \right. \nonumber \\
& & \hspace{1.5cm} \left. + \integral{\pi+2\theta-\rho}{2\pi}{0}{\frac{s_{\theta-\rho}}{s_{\theta-\beta_A}}} \right] \cos(\rho) d\rho \nonumber \\
& & \hspace{0.5cm} + \int_{-\theta}^{0} \left[ \integral{0}{2\theta+\rho}{0}{1} + \integral{2\theta+\rho}{\pi-\rho}{0}{-\frac{s_{\theta+\rho}}{s_{\theta-\beta_A}}} \right. \nonumber\\
& &\hspace{1.5cm} \left. \integral{\pi-\rho}{\pi+2\theta-\rho}{0}{1} + \integral{\pi+2\theta-\rho}{2\pi+\rho}{0}{\frac{s_{\theta-\rho}}{s_{\theta-\beta_A}}}  \right. \nonumber \\
& & \hspace{1.5cm} \left.  \integral{2\pi+\rho}{2\pi}{0}{1} \right] \cos(\rho) d\rho. \nonumber \\
\end{eqnarray}

\begin{eqnarray}
& & M_2 = \int_{\theta}^{2\theta} \left[ \integral{\rho}{2\theta+\rho}{\frac{s_{\theta-\rho}}{s_{\theta-\beta_A}}}{1} + \integral{2\theta+\rho}{\pi-\rho}{\frac{s_{\theta-\rho}}{s_{\theta-\beta_A}}}{-\frac{s_{\theta+\rho}}{s_{\theta-\beta_A}}} \right. \nonumber \\
& & \hspace{1.5cm} \left. + \integral{\pi-\rho}{\pi-\rho+2\theta}{\frac{s_{\theta-\rho}}{s_{\theta-\beta_A}}}{1}   \right] \cos(\rho)\nonumber \\
& & \hspace{0.5cm} + \int_{-2\theta}^{-\theta} \left[ \integral{0}{2\theta+\rho}{-\frac{s_{\theta+\rho}}{s_{\theta-\beta_A}}}{1} + \integral{\pi-\rho}{\pi-\rho+2\theta}{-\frac{s_{\theta+\rho}}{s_{\theta-\beta_A}}}{1} \right. \nonumber \\
& & \hspace{1.5cm} \left. + \integral{\pi-\rho+2\theta}{2 \pi+ \rho}{-\frac{s_{\theta+\rho}}{s_{\theta-\beta_A}}}{\frac{s_{\theta - \rho}}{s_{\theta - \beta_A}}} + \integral{2\pi+\rho}{2\pi}{-\frac{s_{\theta + \rho}}{s_{\theta - \beta_A}}}{1} \right] \cos(\rho) d\rho. \nonumber \\
\end{eqnarray}

\begin{eqnarray}
& & M_3 = \int_{2\theta}^{\pi/2-\theta} \left[ \integral{\rho}{2\theta+\rho}{\frac{s_{\theta - \rho}}{s_{\theta - \beta_A}}}{1} + \integral{2\theta+\rho}{\pi-\rho}{\frac{s_{\theta - \rho}}{s_{\theta - \beta_A}}}{-\frac{s_{\theta + \rho}}{s_{\theta - \beta_A}}} \right. \nonumber \\
& & \hspace{1.5cm} \left. + \integral{\pi-\rho}{\pi+2\theta-\rho}{\frac{s_{\theta-\rho}}{s_{\theta-\beta_A}}}{1} \right] \cos(\rho) d\rho \nonumber \\
& & \hspace{0.5cm} + \int_{-\pi/2+\theta}^{-2\theta} \left[ \integral{\pi-\rho}{\pi+2\theta-\rho}{-\frac{s_{\theta+\rho}}{s_{\theta-\beta_A}}}{1} + \integral{\pi+2\theta-\rho}{2\pi + \rho}{-\frac{s_{\theta + \rho}}{s_{\theta - \beta_A}}}{\frac{s_{\theta-\rho}}{s_{\theta-\beta_A}}} \right. \nonumber \\
& & \hspace{1.5cm} \left. + \integral{2\pi + \rho}{2\pi + 2\theta + \rho}{-\frac{s_{\theta+\rho}}{s_{\theta-\beta_A}}}{1} \right] \cos(\rho) d\rho. \nonumber \\
\end{eqnarray}

\begin{eqnarray}
& & M_4 = \int_{\pi/2-\theta}^{\pi/2} \left[ \integral{\rho}{\pi-\rho}{\frac{s_{\theta - \rho}}{s_{\theta - \beta_A}}}{1} + \integral{\pi-\rho}{2\theta+\rho}{\frac{s_{\theta - \rho}}{s_{\theta - \beta_A}}}{-\frac{s_{\theta + \rho}}{s_{\theta - \beta_A}}} \right. \nonumber \\
& & \hspace{1.5cm} \left. + \integral{2\theta+\rho}{\pi+2\theta-\rho}{\frac{s_{\theta-\rho}}{s_{\theta-\beta_A}}}{1} \right] \cos(\rho) d\rho \nonumber \\
& & \hspace{0.5cm} + \int_{-\pi/2}^{-\pi/2+\theta} \left[ \integral{\pi-\rho}{2\pi+\rho}{-\frac{s_{\theta+\rho}}{s_{\theta-\beta_A}}}{1} + \integral{2\pi+\rho}{\pi+2\theta-\rho}{-\frac{s_{\theta + \rho}}{s_{\theta - \beta_A}}}{\frac{s_{\theta-\rho}}{s_{\theta-\beta_A}}} \right. \nonumber \\
& & \hspace{1.5cm} \left. + \integral{\pi+2\theta-\rho}{2\pi+2\theta+\rho}{-\frac{s_{\theta+\rho}}{s_{\theta-\beta_A}}}{1} \right] \cos(\rho) d\rho. \nonumber \\
\end{eqnarray}

For $\theta\in[\pi/6,\pi/4]$
\begin{eqnarray}
& & I_1 = \int_{0}^{\theta} \left[ \integral{0}{\rho}{0}{\frac{s_{\theta-\rho}}{s_{\theta-\beta_A}}} + \integral{\rho}{2\theta + \rho}{0}{1} \right. \nonumber \\
& & \hspace{1.5cm} \left. + \integral{2\theta+\rho}{\pi-\rho}{0}{-\frac{s_{\theta+\rho}}{s_{\theta-\beta_A}}} + \integral{\pi- \rho}{\pi + 2\theta - \rho}{0}{1}  \right. \nonumber\\
& & \hspace{1.5cm} \left. + \integral{\pi + 2\theta - \rho}{2\pi}{0}{\frac{s_{\theta- \rho}}{s_{\theta - \beta_A}}} \right] \cos(\rho) d\rho \nonumber \\
& & \hspace{0.5cm} + \int_{-\theta}^{0} \left[ \integral{0}{2\theta + \rho}{0}{1} + \integral{2\theta + \rho}{\pi-\rho}{0}{-\frac{s_{\theta+\rho}}{s_{\theta-\beta_A}}} \right. \nonumber \\
& & \hspace{1.5cm} \left. + \integral{\pi-\rho}{\pi+2\theta-\rho}{0}{1} + \integral{\pi+2\theta -\rho}{2\pi+\rho}{0}{\frac{s_{\theta-\rho}}{s_{\theta-\beta_A}}}\right. \nonumber \\
& & \hspace{1.5cm} \left. + \integral{2\pi+\rho}{2\pi}{0}{1} \right] \cos(\rho) d\rho. \nonumber \\
\end{eqnarray}

\begin{eqnarray}
& & I_2 = \int_{\theta}^{\pi/2-\theta} \left[ \integral{\rho}{\rho+2\theta}{\frac{s_{\theta-\rho}}{s_{\theta-\beta_A}}}{1} + \integral{\rho+2\theta}{\pi-\rho}{\frac{s_{\theta-\rho}}{s_{\theta-\beta_A}}}{-\frac{s_{\theta+\rho}}{s_{\theta-\beta_A}}} \right. \nonumber \\
& & \hspace{1.5cm} \left. + \integral{\pi-\rho}{\pi+2\theta-\rho}{\frac{s_{\theta-\rho}}{s_{\theta-\beta_A}}}{1} \right] \cos(\rho) d\rho \nonumber \\
& & \hspace{0.5cm} + \int_{-\pi/2+\theta}^{-\theta} \left[ \integral{0}{2\theta+\rho}{-\frac{s_{\theta+\rho}}{s_{\theta-\beta_A}}}{1} + \integral{\pi-\rho}{\pi+2\theta-\rho}{-\frac{s_{\theta+\rho}}{s_{\theta-\beta_A}}}{1} \right. \nonumber\\
& & \hspace{1.5cm} \left. +�\integral{\pi+2\theta-\rho}{2\pi+\rho}{-\frac{s_{\theta+\rho}}{s_{\theta-\beta_A}}}{\frac{s_{\theta-\rho}}{s_{\theta-\beta_A}}} + \integral{2\pi+\rho}{2\pi}{-\frac{s_{\theta+\rho}}{s_{\theta-\beta_A}}}{1}  \right] \cos(\rho) d\rho. \nonumber \\
\end{eqnarray}

\begin{eqnarray}
& & I_3 = \int_{\pi/2-\theta}^{2\theta} \left[ \integral{\rho}{\pi-\rho}{\frac{s_{\theta-\rho}}{s_{\theta-\beta_A}}}{1}  + \integral{\pi-\rho}{2\theta+\rho}{\frac{s_{\theta-\rho}}{s_{\theta-\beta_A}}}{-\frac{s_{\theta+\rho}}{s_{\theta-\beta_A}}} \right. d\rho \nonumber \\
& & \hspace{1.5cm} \left. + \integral{2\theta+\rho}{\pi+2\theta-\rho}{\frac{s_{\theta-\rho}}{s_{\theta-\beta_A}}}{1} \right] \cos(\rho) d\rho \nonumber \\
& & \hspace{0.5cm} + \int_{-2\theta}^{-\pi/2+\theta} \left[ \integral{0}{2\theta+\rho}{-\frac{s_{\theta+\rho}}{s_{\theta-\beta_A}}}{1} + \integral{\pi-\rho}{2\pi+\rho}{-\frac{s_{\theta+\rho}}{s_{\theta-\beta_A}}}{1} \right. \nonumber \\
& & \hspace{1.5cm} \left. + \integral{2\pi+\rho}{\pi+2\theta-\rho}{-\frac{s_{\theta+\rho}}{s_{\theta-\beta_A}}}{\frac{s_{\theta-\rho}}{s_{\theta-\beta_A}}} + \integral{\pi+2\theta-\rho}{2\pi}{-\frac{s_{\theta+\rho}}{s_{\theta-\beta_A}}}{1}  \right] \cos(\rho) d\rho. \nonumber\\
\end{eqnarray}

\begin{eqnarray}
& & I_4= \int_{2\theta}^{\pi/2} \left[ \integral{\rho}{\pi-\rho}{\frac{s_{\theta-\rho}}{s_{\theta-\beta_A}}}{1} + \integral{\pi-\rho}{2\theta+\rho}{\frac{s_{\theta-\rho}}{s_{\theta-\beta_A}}}{-\frac{s_{\theta+\rho}}{s_{\theta-\beta_A}}} \right. \nonumber \\
& & \hspace{1.5cm} \left. + \integral{2\theta+\rho}{2\theta+\pi-\rho}{\frac{s_{\theta-\rho}}{s_{\theta-\beta_A}}}{1} \right] \cos(\rho) d\rho \nonumber \\
& & \hspace{0.5cm} + \int_{-\pi/2}^{-2\theta} \left[ \integral{\pi-\rho}{2\pi+\rho}{-\frac{s_{\theta+\rho}}{s_{\theta-\beta_A}}}{1} + \integral{2\pi+\rho}{\pi+2\theta-\rho}{-\frac{s_{\theta+\rho}}{s_{\theta-\beta_A}}}{\frac{s_{\theta-\rho}}{s_{\theta-\beta_A}}} \right. \nonumber \\
& & \hspace{1.5cm} \left. + \integral{\pi+2\theta-\rho}{2\pi+2\theta+\rho}{-\frac{s_{\theta+\rho}}{s_{\theta-\beta_A}}}{1} \right]  \cos(\rho) d\rho. \nonumber \\
\end{eqnarray}

\newpage
For $\theta\in[\pi/4,\pi/3]$, 

\begin{eqnarray}
& & J_1 = \int_{0}^{\pi/2-\theta} \left[ \integral{0}{\rho}{0}{\frac{s_{\theta-\rho}}{s_{\theta-\beta_A}}} + \integral{\rho}{2\theta+\rho}{0}{1} \right. d\rho \nonumber \\
& & \hspace{1.5cm} \left. + \integral{2\theta+\rho}{\pi-\rho}{0}{-\frac{s_{\theta+\rho}}{s_{\theta-\beta_A}}}  + \integral{\pi-\rho}{\pi+2\theta-\rho}{0}{1} \right. \nonumber \\
& & \hspace{1.5cm} \left. + \integral{\pi+2\theta-\rho}{2\pi}{0}{\frac{s_{\theta-\rho}}{s_{\theta-\beta_A}}} \right] \cos(\rho) d\rho \nonumber \\
& & \hspace{0.5cm} + \int_{\theta-\pi/2}^{0} \left[�\integral{0}{2\theta+\rho}{0}{1} + \integral{2\theta+\rho}{\pi-\rho}{0}{-\frac{s_{\theta+\rho}}{s_{\theta-\beta_A}}} \right. d\rho \nonumber \\
& & \hspace{1.5cm} \left. + \integral{\pi-\rho}{\pi+2\theta-\rho}{0}{1} + \integral{\pi+2\theta-\rho}{2\pi+\rho}{0}{\frac{s_{\theta-\rho}}{s_{\theta-\beta_A}}} \right. \nonumber \\
& & \hspace{1.5cm} \left. + \integral{2\pi+\rho}{2\pi}{0}{1} \right] \cos(\rho) d\rho. \nonumber \\ 
\end{eqnarray}

\begin{eqnarray}
& & J_2 = \int_{\pi/2-\theta}^{\theta} \left[ \integral{0}{\rho}{0}{\frac{s_{\theta-\rho}}{s_{\theta-\beta_A}}} + \integral{\rho}{\pi-\rho}{0}{1} \right. \nonumber \\
& & \hspace{1.5cm} \left. + \integral{\pi-\rho}{2\theta+\rho}{0}{-\frac{s_{\theta+\rho}}{s_{\theta-\beta_A}}} + \integral{2\theta+\rho}{\pi+2\theta-\rho}{0}{1} \right. \nonumber \\
& & \hspace{1.5cm} \left. + \integral{\pi+2\theta-\rho}{2\pi}{0}{\frac{s_{\theta-\rho}}{s_{\theta-\beta_A}}} \right] \cos(\rho) d\rho \nonumber \\
& & \hspace{0.5cm} + \int_{-\theta}^{\theta-\pi/2} \left[ \integral{0}{2\theta+\rho}{0}{1} + \integral{\theta+\rho}{\pi-\rho}{0}{-\frac{s_{\theta+\rho}}{s_{\theta-\beta_A}}} \right. \nonumber \\
& & \hspace{1.5cm} \left. + \integral{\pi-\rho}{2\pi+\rho}{0}{1} + \integral{2\pi+\rho}{\pi+2\theta-\rho}{0}{\frac{s_{\theta-\rho}}{s_{\theta-\beta_A}}} \right. \nonumber \\
& & \hspace{1.5cm} + \integral{\pi+2\theta-\rho}{2\pi}{0}{1} \left. \right] \cos(\rho) d\rho. 
\end{eqnarray}

\begin{eqnarray}
& & J_3 = \int_{\theta}^{\pi-2\theta} \left[ \integral{\rho}{\pi-\rho}{\frac{s_{\theta-\rho}}{s_{\theta-\beta_A}}}{1} + \integral{\pi-\rho}{2\theta+\rho}{\frac{s_{\theta-\rho}}{s_{\theta-\beta_A}}}{-\frac{s_{\theta+\rho}}{s_{\theta-\beta_A}}} \right. \nonumber \\
& & \hspace{1.5cm} \left. + \integral{2\theta+\rho}{\pi+2\theta-\rho}{\frac{s_{\theta-\rho}}{s_{\theta-\beta_A}}}{1} \right] \cos(\rho) d\rho \nonumber \\
& & \hspace{0.5cm} + \int_{2\theta-\pi}^{-\theta} \left[ \integral{0}{2\theta+\rho}{-\frac{s_{\theta+\rho}}{s_{\theta-\beta_A}}}{1} +\integral{\pi-\rho}{2\pi+\rho}{-\frac{s_{\theta+\rho}}{s_{\theta-\beta_A}}}{1} \right. \nonumber \\ 
& & \hspace{1.5cm} \left. + \integral{2\pi+\rho}{\pi+2\theta-\rho}{-\frac{s_{\theta+\rho}}{s_{\theta-\beta_A}}}{\frac{s_{\theta-\rho}}{s_{\theta-\beta_A}}} + \integral{\pi+2\theta-\rho}{2\pi}{-\frac{s_{\theta+\rho}}{s_{\theta-\beta_A}}}{1} \right] \cos(\rho) d\rho. \nonumber \\
\end{eqnarray}

\begin{eqnarray}
& & J_4 = \int_{\pi-2\theta}^{\pi/2} \left[ \integral{\rho}{\pi-\rho}{\frac{s_{\theta-\rho}}{s_{\theta-\beta_A}}}{1} + \integral{\pi-\rho}{2\theta+\rho}{\frac{s_{\theta-\rho}}{s_{\theta-\beta_A}}}{-\frac{s_{\theta+\rho}}{s_{\theta-\beta_A}}}\right. \nonumber\\
& & \hspace{1.5cm} \left. + \integral{2\theta+\rho}{\pi+2\theta-\rho}{\frac{s_{\theta-\rho}}{s_{\theta-\beta_A}}}{1} \right] \cos(\theta) d\rho \nonumber \\
& & \hspace{0.5cm} + \int_{-\pi/2}^{2\theta-\pi} \left[ \integral{0}{2\theta-\rho-\pi}{-\frac{s_{\theta+\rho}}{s_{\theta-\beta_A}}}{\frac{s_{\theta-\rho}}{s_{\theta-\beta_A}}} + \integral{2\theta-\rho-\pi}{2\theta+\rho}{-\frac{s_{\theta+\rho}}{s_{\theta-\beta_A}}}{1} \right. \nonumber \\
& & \hspace{1.5cm} \left. + \integral{\pi-\rho}{2\pi+\rho}{-\frac{s_{\theta+\rho}}{s_{\theta-\beta_A}}}{1} + \integral{2\pi+\rho}{2\pi}{-\frac{s_{\theta+\rho}}{s_{\theta-\beta_A}}}{\frac{s_{\theta-\rho}}{s_{\theta-\beta_A}}} \right] \cos(\rho) d\rho. \nonumber \\
\end{eqnarray}

\newpage
For $\theta\in[\pi/3,\pi/2]$

\begin{eqnarray}
& & L_1 = \int_{0}^{\pi/2-\theta} \left[ \integral{0}{\rho}{0}{\frac{s_{\theta-\rho}}{s_{\theta-\beta_A}}} + \integral{\rho}{2\theta+\rho}{0}{1} \right. \nonumber \\
& & \hspace{1.5cm} \left. + \integral{2\theta+\rho}{\pi-\rho}{0}{-\frac{s_{\theta+\rho}}{s_{\theta-\beta_A}}} + \integral{\pi-\rho}{\pi+2\theta-\rho}{0}{1} \right. \nonumber \\
& & \hspace{1.5cm} \left. + \integral{\pi+2\theta-\rho}{2\pi}{0}{\frac{s_{\theta-\rho}}{s_{\theta-\beta_A}}} \right] \cos(\rho) d\rho \nonumber \\
& &  \hspace{0.5cm} + \int_{\theta-\pi/2}^{0}\left[ \integral{0}{2\theta+\rho}{0}{1} + \integral{2\theta+\rho}{\pi-\rho}{0}{-\frac{s_{\theta+\rho}}{s_{\theta-\beta_A}}} \right. \nonumber \\
& & \hspace{1.5cm} \left. + \integral{\pi-\rho}{\pi+2\theta-\rho}{0}{1} + \integral{\pi+2\theta-\rho}{2\pi+\rho}{0}{\frac{s_{\theta-\rho}}{s_{\theta-\beta_A}}} \right. \nonumber \\
& & \hspace{1.5cm} \left. + \integral{2\pi+\rho}{2\pi}{0}{1} \right] \cos(\rho) d\rho. \nonumber \\
\end{eqnarray}

\begin{eqnarray}
& & L_2 = \int_{\pi/2-\theta}^{\pi-2\theta} \left[ \integral{0}{\rho}{0}{\frac{s_{\theta-\rho}}{s_{\theta-\beta_A}}} + \integral{\rho}{\pi-\rho}{0}{1} \right. \nonumber \\
& & \hspace{1.5cm} \left. + \integral{\pi-\rho}{2\theta-\rho}{0}{-\frac{s_{\theta+\rho}}{s_{\theta-\beta_A}}} + \integral{2\theta+\rho}{\pi+2\theta-\rho}{0}{1} \right.\nonumber \\
& &\hspace{1.5cm} \left. +\integral{\pi+2\theta-\rho}{2\pi}{0}{\frac{s_{\theta-\rho}}{s_{\theta-\beta_A}}} \right] \cos(\rho) d\rho \nonumber \\
& & \hspace{0.5cm} + \int_{2\theta-\pi}^{\theta-\pi/2} \left[ \integral{0}{2\theta+\rho}{0}{1} + \integral{2\theta+\rho}{\pi-\rho}{0}{-\frac{s_{\theta+\rho}}{s_{\theta-\beta_A}}} \right. \nonumber\\
& &\hspace{1.5cm} \left. \integral{\pi-\rho}{2\pi+\rho}{0}{1} + \integral{2\pi+\rho}{\pi+2\theta-\rho}{0}{\frac{s_{\theta-\rho}}{s_{\theta-\beta_A}}} \right. \nonumber \\
& & \hspace{1.5cm} \left. + \integral{\pi+2\theta-\rho}{2\pi}{0}{1} \right] \cos(\rho) d\rho. \nonumber \\\end{eqnarray}

\begin{eqnarray}
& & L_3 = \int_{\pi-2\theta}^{\theta} \left[ \integral{0}{\rho}{0}{\frac{s_{\theta-\rho}}{s_{\theta-\beta_A}}} + \integral{\rho}{\pi-\rho}{0}{1} \right. \nonumber \\
& & \hspace{1.5cm} \left. + \integral{\pi-\rho}{2\theta+\rho}{0}{-\frac{s_{\theta+\rho}}{s_{\theta-\beta_A}}} + \integral{2\theta+\rho}{\pi+2\theta-\rho}{0}{1} \right. \nonumber \\
& & \hspace{1.5cm} \left. + \integral{\pi+2\theta-\rho}{2\pi}{0}{\frac{s_{\theta-\rho}}{s_{\theta-\beta_A}}} \right] \cos(\rho) d\rho \nonumber \\
& &  \hspace{0.5cm} + \int_{-\theta}^{2\theta-\pi}\left[ \integral{0}{2\theta-\pi-\rho}{0}{\frac{s_{\theta-\rho}}{s_{\theta-\beta_A}}} + \integral{2\theta-\pi-\rho}{2\theta+\rho}{0}{1} \right. \nonumber \\
& & \hspace{1.5cm} \left. + \integral{2\theta+\rho}{\pi-\rho}{0}{-\frac{s_{\theta+\rho}}{s_{\theta-\beta_A}}} + \integral{\pi-\rho}{2\pi+\rho}{0}{1} \right. \nonumber \\
& & \hspace{1.5cm} \left. + \integral{2\pi+\rho}{2\pi}{0}{\frac{s_{\theta-\rho}}{s_{\theta-\beta_A}}} \right] \cos(\rho) d\rho. \nonumber \\
\end{eqnarray}

\begin{eqnarray}
& & L_4 = \int_{\theta}^{\pi/2} \left[ \integral{\rho}{\pi-\rho}{\frac{s_{\theta-\rho}}{s_{\theta-\beta_A}}}{1} + \integral{\pi-\rho}{2\theta+\rho}{\frac{s_{\theta-\rho}}{s_{\theta-\beta_A}}}{-\frac{s_{\theta+\rho}}{s_{\theta-\beta_A}}} \right. \nonumber \\
& & \hspace{1.5cm} \left. + \integral{2\theta+\rho}{\pi+2\theta-\rho}{\frac{s_{\theta-\rho}}{s_{\theta-\beta_A}}}{1} \right] \cos(\rho) d\rho \nonumber \\
& & \hspace{0.5cm} + \int_{-\pi/2}^{-\theta} \left[ \integral{0}{2\theta-\pi-\rho}{-\frac{s_{\theta+\rho}}{s_{\theta-\beta_A}}}{\frac{s_{\theta-\rho}}{s_{\theta-\beta_A}}} + \integral{2\theta-\pi-\rho}{2\theta+\rho}{-\frac{s_{\theta+\rho}}{s_{\theta-\beta_A}}}{1} \right. \nonumber\\
& & \hspace{1.5cm} \left. + \integral{\pi-\rho}{2\pi+\rho}{-\frac{s_{\theta+\rho}}{s_{\theta-\beta_A}}}{1} +\integral{2\pi+\rho}{2\pi}{-\frac{s_{\theta+\rho}}{s_{\theta-\beta_A}}}{\frac{s_{\theta-\rho}}{s_{\theta-\beta_A}}} \right] \cos(\rho) d\rho. \nonumber \\
\end{eqnarray}


\bibliographystyle{plain}
\bibliography{01_references}

\end{document}